\documentclass[11pt]{amsart}
\usepackage[T1]{fontenc}
\usepackage[english]{babel}
\usepackage{amscd,amsmath,amsthm,amssymb,graphics}
\usepackage{lmodern,pst-node}
\usepackage{pstcol,pst-plot,pst-3d}
\usepackage{multicol}
\usepackage{epic,eepic}
\usepackage{amsfonts,amssymb,amscd,amsmath,enumitem,verbatim}
\psset{unit=0.7cm,linewidth=0.8pt,arrowsize=2.5pt 4}

\newpsstyle{fatline}{linewidth=1.5pt}
\newpsstyle{fyp}{fillstyle=solid,fillcolor=verylight}
\definecolor{verylight}{gray}{0.97}
\definecolor{light}{gray}{0.9}
\definecolor{medium}{gray}{0.85}
\definecolor{dark}{gray}{0.6}

\def\frk{\frak}               

\def\Phi{{\frk n}}
\def\Phi{{\frk N}}
\def\opn#1#2{\def#1{\operatorname{#2}}} 
\opn\chara{char} \opn\length{\ell} \opn\pd{pd} \opn\rk{rk}
\opn\projdim{proj\,dim} \opn\injdim{inj\,dim} \opn\rank{rank}
\opn\depth{depth} \opn\grade{grade} \opn\height{height}
\opn\embdim{emb\,dim} \opn\codim{codim}

\opn\Tr{Tr} \opn\bigrank{big\,rank}
\opn\superheight{superheight}\opn\lcm{lcm}
\opn\trdeg{tr\,deg}
\opn\reg{reg} \opn\lreg{lreg} \opn\ini{in} \opn\lpd{lpd}
\opn\size{size}\opn\bigsize{bigsize}
\opn\cosize{cosize}\opn\bigcosize{bigcosize}
\opn\sdepth{sdepth}\opn\sreg{sreg}
\opn\link{link}\opn\fdepth{fdepth}
\opn\deg{deg}
\opn\max{max}
\opn\indeg{indeg}
\opn\min{min}
\opn\psln{psln}
\opn\div{div} \opn\Div{Div} \opn\cl{cl} \opn\Cl{Cl}
\let\epsilon\varepsilon
\let\phi=\varphi
\let\kappa=\varkappa

\opn\Spec{Spec} \opn\Supp{Supp} \opn\supp{supp} \opn\Sing{Sing}
\opn\Ass{Ass} \opn\Min{Min}\opn\Mon{Mon} \opn\dstab{dstab} \opn\astab{astab}
\opn\Syz{Syz}
\opn\Ann{Ann} \opn\Rad{Rad} \opn\Soc{Soc}
\opn\Im{Im} \opn\Ker{Ker} \opn\Coker{Coker} \opn\Am{Am}
\opn\Hom{Hom} \opn\Tor{Tor} \opn\Ext{Ext} \opn\End{End}
\opn\Aut{Aut} \opn\id{id}

\opn\nat{nat}
\opn\pff{pf}
\opn\Pf{Pf} \opn\GL{GL} \opn\SL{SL} \opn\mod{mod} \opn\ord{ord}
\opn\Gin{Gin} \opn\Hilb{Hilb}\opn\sort{sort}
\opn\initial{init}
\opn\ende{end}
\opn\height{height}
\opn\bight{bight}
\opn\hte{ht}
\opn\indeg{indeg}
\opn\reg{reg}
\opn\depth{depth}
\opn\type{type}
\opn\ldim{ldim}
\opn\maxdeg{maxdeg}
\opn\aff{aff} \opn\con{conv} \opn\relint{relint} \opn\st{st}
\opn\lk{lk} \opn\cn{cn} \opn\core{core} \opn\vol{vol}
\opn\link{link} \opn\star{star}\opn\lex{lex}
\opn\gr{gr}

\def\pot#1#2{#1[\kern-0.28ex[#2]\kern-0.28ex]}

\opn\dirlim{\underrightarrow{\lim}}
\opn\inivlim{\underleftarrow{\lim}}
\def\Implies{\ifmmode\Longrightarrow \else
        \unskip${}\Longrightarrow{}$\ignorespaces\fi}
\def\implies{\ifmmode\Rightarrow \else
        \unskip${}\Rightarrow{}$\ignorespaces\fi}
\def\iff{\ifmmode\Longleftrightarrow \else
        \unskip${}\Longleftrightarrow{}$\ignorespaces\fi}

\let\:=\colon
 \theoremstyle{plain}
\newtheorem{Theorem}{Theorem}[section]
 \newtheorem{Lemma}[Theorem]{Lemma}
 \newtheorem{Corollary}[Theorem]{Corollary}
 \newtheorem{Proposition}[Theorem]{Proposition}

 \theoremstyle{definition}
 \newtheorem{Definition}[Theorem]{Definition}
 \newtheorem{Remark}[Theorem]{Remark}
 
 \newtheorem{Example}[Theorem]{Example}
 
 \newtheorem{acknowledgement}{Acknowledgement}

\let\epsilon\varepsilon
\let\kappa=\varkappa
\def\qed{\ifhmode\textqed\fi
      \ifmmode\ifinner\quad\qedsymbol\else\dispqed\fi\fi}
\def\textqed{\unskip\nobreak\penalty50
       \hskip2em\hbox{}\nobreak\hfil\qedsymbol
       \parfillskip=0pt \finalhyphendemerits=0}
\def\dispqed{\rlap{\qquad\qedsymbol}}

\opn\dis{dis}
\def\pnt{{\raise0.5mm\hbox{\large\bf.}}}

\opn\Lex{Lex}

\begin{document}

\author[Mafi and Naderi]{ Amir Mafi and Dler Naderi}
\title{A note on almost Cohen-Macaulay monomial ideals}

\address{Amir Mafi, Department of Mathematics, University Of Kurdistan, P.O. Box: 416, Sanandaj, Iran.}
\email{A\_Mafi@ipm.ir}
\address{Dler Naderi, Department of Mathematics, University of Kurdistan, P.O. Box: 416, Sanandaj,
Iran.}
\email{dler.naderi65@gmail.com}

\begin{abstract}
Let $R = k[x_1,\ldots, x_n]$ be the polynomial ring in $n$ variables over a field $k$ and let $I$ be a monomial ideal of $R$. In this paper, we study almost Cohen-Macaulay simplicial complex. Moreover, we characterize the almost Cohen-Macaulay polymatroidal Veronese type and transversal polymatroidal ideals and furthermore we give some examples.
\end{abstract}

\subjclass[2010]{ 13C14, 13C05}
\keywords{Almost Cohen-Macaulay rings, polymatroidal ideals, simplicial complex, Veronese type ideal}

\maketitle
\section*{Introduction}
Throughout this paper, we assume that $R = k[x_1,\ldots, x_n]$ is the polynomial ring in $n$ variables over a field $k$, $\frak{m}=(x_1,\ldots,x_n)$ the unique homogeneous maximal ideal of $R$ and $I$ a monomial ideal of $R$. We denote, as usual, by $G(I)$ the unique minimal set of monomial generators of $I$. If $I$ is generated in a single degree, then $I$ is said to be {\it polymatroidal} if for any two elements $u, v \in G(I)$ such that $\deg_{x_{i}}(v) < \deg_{ x_{i}}(u)$ there exists an index $j$ with $\deg_{x_{j}} (u) < \deg_{x_{j}} (v)$ such that $x_{j}(u/x_{i}) \in G(I)$. The polymatroidal ideal $I$ is called {\it matroidal} if $I$ is generated by square-free monomials (see \cite{HH2} or \cite{HH1}). One of the most distinguished polymatroidal ideals is the ideal of {\it Veronese type} and the other is {\it transversal} polymatroidal ideals. Consider the fixed positive integers $d$ and $1\leq a_n \leq\ldots\leq a_1 \leq d$. The ideal of Veronese type of $R$ indexed by $d$ and $(a_1,\ldots, a_n)$ is the ideal $I_{(d; a_1,\ldots, a_n)}$ which is generated by those monomials $u = x_1^{i_1}\ldots x_n^{i_n}$ of $R$ of degree $d$ with $i_j \leq a_j$ for each $1 \leq j \leq n$. Let $F$ be a non-empty subset of $[n]$ and $P_F = ({x_i | i \in F})$ is the monomial prime ideal. A transversal polymatroidal ideal is an ideal $I$ of the form $I = P_{F_1} P_{F_2}\ldots P_{F_r}$, where $F_1,\ldots, F_r$ is a collection of non-empty subsets of $[n]$ with $r \geq 1$ (see \cite{HRV}).

Herzog and Hibi \cite{HH3} proved that a polymatroidal ideal $I$ is Cohen-Macaulay (i.e. CM) if and only if $I$ is a principal ideal, a Veronese ideal, or a square-free Veronese ideal. Note that $I$ is CM whenever $R/I$ is a CM ring. Vladoiu in \cite{VL} proved that a Veronese type ideal $I$ is CM if and only if $\Ass(I)=\Min(I)$. We say that the monomial ideal $I$ is almost Cohen-Macaulay (i.e. aCM) when $\depth R/I \geq \dim R/I-1$. It is clear that all CM monomial ideals are aCM. Several authors studied almost Cohen-Macaulay modules (see for example \cite{CTT, Ha, I, K1, K2, MT, TM, TMA}).\\
For a square-free monomial ideal $I$ of $R$, we may consider the simplicial complex $\Delta$ for which $I = I_{\Delta}$ is the Stanley-Reisner ideal of $\Delta$ and $K[\Delta] = R/I_{\Delta}$ is the Stanley-Reisner ring. Eagon and Reiner \cite{ER} proved that $I$ is CM if and only if the square-free Alexander dual $I^{\vee}$ has linear resolution.

In this paper we are interested in studying the aCM simplicial complex. Also, we characterize the aCM polymatroidal Veronese type and transversal polymatroidal ideals and we give some examples. For any unexplained notion or terminology, we refer the reader to
\cite{HH1, V}. Several explicit examples were performed with help of the computer algebra system Macaulay2 \cite{G}.

\section{Preliminaries}
In this section, we recall some definitions and known results which is used in this paper. Let $\Delta$ be a simplicial complex on the vertex set $V = \{x_1,\ldots, x_n \}$. Every element of $\Delta$ is called a face of $\Delta$ and a facet of $\Delta$ is a maximal face of $\Delta$ with respect to inclusion. If all facets of $\Delta$ have the same cardinality, then $\Delta$ is called pure. Let $\Delta^{\vee}$ be the dual simplicial complex of $\Delta$, that is to say, $\Delta^{\vee} = \{V \setminus F~ | ~F \notin \Delta \}$. If $I$ is a square-free monomial ideal, then $I = \cap_{i=1}^{t} \frk{p_{i}}$ where each of the $\frk{p_{i}}$ is a monomial prime ideal of $I$. The ideal $I^{\vee}$ which is minimally generated by the monomial $u_i = \prod_{x_{j} \in \frk{p_{i}}} x_{j}$ is called the {\it Alexander dual} of $I$. For the simplicial complex $\Delta$ and $F \in \Delta$, link of $F$ in $\Delta$ is defined as $lk_{\Delta}(F) = \{G \in \Delta ~| ~G \cap  F = \emptyset, G \cup  F \in \Delta \}$. Let $M$ be a finitely generated graded $R$-module, the regularity of $M$ is defined by
\[\reg M := \max \{ j ~| ~ \beta_{i,i+j}(M) \ne 0 \}.\]
Terai \cite{Te}, defined the initial degree of $M$ by
\[ \indeg M:= \min \{ j ~|~ M_{j} \neq 0 \}= \min \{ j ~|~ \beta_{0,j}(M) \neq 0 \}.\]
It is clear that $\reg M \geq  \indeg M$, with equality if and only if $M$ has linear resolution.
Also, Terai proved the following interesting results:
\begin{Theorem}\label{T1}
 Let $\Delta$ be a simplicial complex on the vertex set $[n]$ . Then
\[ \reg (I_{\Delta})-\indeg  (I_{\Delta})= \dim k[\Delta^{\vee}] - \depth k[\Delta^{\vee}].\]
In particular $\reg I_{\Delta}= \pd k[\Delta^{\vee}]$ and $\indeg I_{\Delta} = \mathrm{e}\mathrm{m}\mathrm{b} \dim k[\Delta^{\vee}]-\dim k[\Delta^{\vee}]$.
\end{Theorem}

\begin{Theorem}\label{T0}(\cite[Theorem 3.4]{HT})
 Let $I=(u_1,\ldots,u_r)$ be a monomial ideal of $R$. If $u_1,\ldots,u_r$ is an $R$-regular sequence with $\deg(u_i)=d_i$, then $\reg(I)=d_1+\ldots+d_r-r+1$.
\end{Theorem}
Herzog, Rauf and Vladoiu in \cite{HRV} defined the following definition:

\begin{Definition}
Let $I$ be a transversal polymatroidal ideal of the form $I = P_{F_1} P_{F_2}\ldots P_{F_r}$. The graph $G_I$ associated with $I$ is defined as follows:
the set of vertices $\mathcal{V}(G_I)$ is the set $\{1,\ldots,r\}$ and $\{i,j\}$ is an edge of $G_I$ if and only if $F_i\cap F_j\neq\emptyset$.
\end{Definition}

\begin{Theorem}\label{T5}(\cite[Theorem 4.7]{HRV})
Let $I$ be a transversal polymatroidal ideal. Then\\
$\Ass(I) = \{ P_{\mathcal{T}} : \mathcal{T}$  is a tree in $G_I \}$.
\end{Theorem}

\begin{Corollary}\label{C3}\cite[Corollary 4.10]{HRV})
Let $I$  be a transversal polymatroidal ideal with the set of associated prime ideals $Ass(I) = \{ P_1,\ldots, P_l \}$. Consider $T_1,\ldots, T_l$ maximal trees of $G_I$ such that $P_j = P_{T_j}$ for all $j = 1,\ldots, l$. Then
\[ I^{k}=\cap_{j=1}^{l}P^{ka_{j}},\]
 is an irredundant primary decomposition of $I^k$ for every $k \geq 1$, where $a_j = |V(T_j)|$ for all $j$.
\end{Corollary}

\begin{Theorem}\label{T6}(\cite[Theorem 4.12]{HRV})
Let $I=P_{F_1} P_{F_2}\ldots P_{F_d}$ be a transversal polymatroidal ideal. Then
\[ \depth(R/I)= c(G_{I})-1+n -|\cup_{i=1}^{d} F_{i} |,\]
where by $c(G_{I})$ we denote the number of connected components of the graph $G_I$ .
\end{Theorem}

Vladoiu in \cite{VL} proved the following interesting result about associated prime ideals of Veronese type:

\begin{Theorem}\label{T7}
Let $I=I_{d;a_1,\ldots,a_n}$ be an ideal of Veronese type with $d>1$ and $a_i\geq 1$ for $i=1,\ldots n$. Then
$P_A\in\Ass(I)\Longleftrightarrow \sum_{i=1}^n a_i\geq d-1+\mid A\mid$ and $\sum_{i\notin A}a_i\leq d-1$.
\end{Theorem}
\section{The aCM polymatroidal ideal}
We start this section by the following definition:
\begin{Definition}
Let $I$ be a square-free monomial ideal in $R$. We say that $I$ has almost linear resolution precisely when $\reg(I) \leq  \indeg (I)+1$.
\end{Definition}
\begin{Remark}
 Let $I$ be a square-free monomial ideal in $R$ and $I^{\vee}$ be the Alexander dual of $I$. Then, by Theorem \ref{T1}, we have
 \[  \dim (R/I)- \depth(R/I)=\reg(I^{\vee}) -  \indeg (I^{\vee}).\]
 Hence $I$  is aCM if and only if $I^{\vee}$ has almost linear resolution.
 \end{Remark}
By the Auslander-Buchsbaum formula, it is known that the monomial ideal $I$ is CM if and only if $\hte (I) = \pd (R/I)$. We extend this result:

\begin{Proposition}\label{P1}
 Let $I$ be a monomial ideal of $R$. Then $I$ is aCM if and only if $\hte (I) \geq \pd (R/I) -1$.
 \end{Proposition}
 \begin{proof}
Suppose that $I$ is aCM. Then $\depth R/I = \dim R/I$ or $\depth R/I = \dim R/I -1$. If $\depth R/I = \dim R/I$, then $I$ is CM and so $\hte (I) = \pd (R/I)$. Thus we have the result in this case. Let $\depth R/I = \dim R/I - 1$. Then by the Auslander-Buchsbaum formula, we have $\pd (R/I) = n- \depth R/I = n- \dim R/I + 1 = \hte (I)+1$. Therefore the result follows.
Conversely, let $\hte (I) \geq  \pd (R/I)-1$. Then $\dim R/I + \hte(I) \geq  \dim R/I + \pd (R/I)- 1$ and so $n- \pd (R/I) \geq \dim R/I -1$. Again by the Auslander-Buchsbaum formula $\depth R/I \geq \dim R/I - 1$. Thus $I$ is aCM, as required.
 \end{proof}
Let $I$ be a monomial ideal of $R$. Then the big height of $I$, denoted by $\bight(I)$, is $\max \{ \height (\frk{p}) ~|~ \frk{p} \in \Ass(I) \}$.

\begin{Corollary} \label{C1}
 Let $I$ be an aCM monomial ideal of $R$. Then $\bight(I) - \hte(I) \leq 1$.
 \end{Corollary}
 \begin{proof}
It is known that $\bight(I) \leq \pd(R/I)$. Thus by using Proposition \ref{P1}, we have $\bight(I) - \hte(I) \leq \pd(R/I) - \hte(I) \leq 1$.
 \end{proof}

\begin{Theorem}\label{T2}
Let $I$ be an ideal of Veronese type. Then the following statements are equivalent:
\item[(i)] $I$ is almost Cohen-Macaulay
\item[(ii)]
 $\bight(I) - \hte(I)\leq 1$

\end{Theorem}
\begin{proof}
$(i) \Rightarrow (ii):$ This is obvious by Corollary \ref{C1}.\\
$(ii) \Rightarrow (i):$ Since $I$ is a Veronese type ideal, by Theorem \ref{T7}, we have
\[ P_{A} \in \Ass(I) \Longleftrightarrow \sum_{i=1}^{n} a_{i}-d+1 \geq \vert A \vert , ~~~ \sum_{i \notin A} a_{i} \leq d-1. \]
If $ \bight(I)  =\sum_{i=1}^{n} a_{i}-d+1$, then $\hte(I) \geq \sum_{i=1}^{n} a_{i}-d$. By \cite[Corollary 5.7]{HRV}, we have $\depth (R/I)= \max \{0, d+n-1-\sum_{i=1}^{n}a_{i} \}$ and so $\depth (R/I) = 0$ or $\depth (R/I) \ne 0$.
If $\depth (R/I) \ne 0$, then $\dim (R/I) =n - \hte (I) \leq n- \sum_{i=1}^{n} a_{i}+d = \depth (R/I)+1$. Thus I is aCM in this case. If $\depth (R/I) =0$, then $\frk{m} \in \Ass(I)$ and so $\hte(I) \geq n -1$. Thus $\dim R/I \leq 1$. Hence $I$ is aCM. Now, suppose that $\bight(I)  \leq\sum_{i=1}^{n} a_{i}-d$. Assume $P_{B}=(x_{i_1},\ldots, x_{i_r}) \in \Ass(I)$ such that $\bight(I) = \hte(P_B)$. Then $P_{C}=(x_{i_1},\ldots, x_{i_r}, x_{i_{r+1}}) $ satisfies the conditions in Theorem \ref{T7}, since $\sum_{i=1}^{n} a_{i}-d+1 \geq \vert C \vert$ and $\sum_{i \notin C} a_{i} \leq \sum_{i \notin B} a_{i} \leq d-1$. Therefore $P_{C} \in \Ass(I)$ and this is contrary to $\bight(I) = \hte(P_B)$. This completes the proof.
\end{proof}

\begin{Corollary}(Compare with \cite[Theorem 3.4]{VL})
Let $I$ be an ideal of Veronese type. Then $I$ is Cohen-Macaulay if and only if  $\bight(I)=\hte(I)$.
\end{Corollary}

\begin{proof}
If $I$ is CM, then it is clear that $\bight(I)=\hte(I)$.
Conversely, suppose that $\bight(I)=\hte(I)$. By \cite[Corollary 5.7]{HRV}, we have $\depth (R/I)= \max \{0, d+n-1-\sum_{i=1}^{n}a_{i} \}$. If $\depth (R/I)=0$, then $\frak{m}\in\Ass(I)$ and so $I=\frak{m}^t$ for some $t\in\mathbb{N}$.
Thus we have the result in this case. Let $\depth R/I\neq 0$ and so $\depth (R/I)= d+n-1-\sum_{i=1}^{n}a_{i}$. On the other hand,  by using the proof of Theorem \ref{T2} we conclude that $ \bight(I)=\sum_{i=1}^{n} a_{i}-d+1$ and so $\dim R/I=n-\hte(I)=n-\sum_{i=1}^{n} a_{i}+d-1=\depth (R/I)$. Therefore $I$ is CM. This completes the proof.
\end{proof}
The following example shows that the above theorem is not true for all square-free transversal polymatroidal ideals.

\begin{Example}
Let $I=(x_1x_4, x_1x_5, x_1x_6, x_1x_7, x_2x_4, x_2x_5, x_2x_6, x_2x_7, x_3x_4, x_3x_5, x_3x_6, x_3x_7)$ be a monomial ideal of $R=k[x_1,\ldots,x_7]$. Then $I$ is matroidal ideal and has the following primary decomposition,
\[ I=(x_1, x_2, x_3 )(x_4,x_5,x_6,x_7).\]
Therefore $\bight(I)-\hte(I) \leq 1$, but $I$ is not aCM.
\end{Example}
In the following result we use the ideal $P_{i}=(x_1, x_2,\ldots ,\widehat {{x_i}},\ldots,x_n)$, where $x_i$ in the generated of $P_{i}$ is omitted.

\begin{Proposition}
Let $I$ be an aCM polymatroidal ideal of degree $d$ such that $\frk{m} \in \Ass(I)$. Then $I$ is a Veronese type ideal.
\end{Proposition}
\begin{proof}
Since $\frk{m} \in \Ass(I)$, then $\depth(R/I)=0$. Since $I$ is aCM, it follows that $\dim (R/I)=1$. Therefore $\hte(I)=n-1$. Since $I$ is a polymatroidal ideal, $I$ has the following presentation, $I=P_{1}^{d_1} \cap P_{2}^{d_2} \cap\ldots \cap  P_{r}^{d_r} \cap \frak{m}^{d}$ where $P_{i}=(x_1, x_2,\ldots ,\widehat {{x_i}},\ldots,x_n)$ and $d_i=\reg(I_{P_{i}})$ for all $i=1,2,\ldots, r$ (see \cite[Theorems 2.4, 2.6]{HV}). Since $I(P_{\{i\}})=I : x_{i}^{a_i} =P_{i}^{d_i} \cap \frak{m}^{d-a_i} =P_{i}^{d-a_i}$ for large $a_i$, this implies that $I$ is a Veronese type ideal by using \cite[Proposition 1.10]{MN}.
\end{proof}

The following example says that the assumption of $\frk{m} \in \Ass(I)$ in the above proposition is essential.
\begin{Example}
Let $I=(x_1x_2,x_1x_3)$ be a polymatroidal ideal of $R=k[x_1, x_2, x_3]$. Then $\dim (R/I)=2$ and $\depth(R/I)=1$  and so $I$ is aCM but $I$ is not a Veronese type ideal.
\end{Example}

In the following example $I$ is aCM, but $\pd(R/I) \ne \bight(I)$.
\begin{Example}
Let $I=(x_1x_3, x_1x_4, x_2x_3, x_2x_4)$ be an ideal of $R=k[x_1, x_2, x_3, x_4]$. Then $I$ is matroidal ideal of $R$ with $\dim (R/I)=2$, $\depth(R/I)=1$ and so $I$ is aCM, but $3=\pd(R/I)\neq\bight(I)=2$.
\end{Example}

For a monomial ideal $I$ of $R$ and $G(I)=\{u_1,\ldots,u_t\}$, we set $\supp(I)=\cup_{i=1}^t\supp(u_i)$, where $\supp(u)=\{x_i\mid u=x_1^{a_1}\ldots x_n^{a_n}, a_i\neq 0\}$ and we say that the monomial ideal $I$ is full-supported if $\supp(I)=\{x_1,\ldots,x_n\}$.

\begin{Theorem}\label{T3}
Let $I$ be a full-supported transversal polymatroidal ideal of degree $d>1$ with $I=P_{F_1} P_{F_2}\ldots P_{F_d}$. Then $I$ is aCM if and only if
\begin{enumerate}
\item[(i)] $I$ is principal ideal.
\item[(ii)]
$I=(x_1\ldots\hat{x_i}\ldots x_n,x_1\ldots\hat{x_j}\ldots x_n)$, where $1\leq i<j\leq n$.
\item[(iii)]
$I=(x_1\ldots x_{i-1}x_i^2x_{i+1}\ldots x_{n-1},x_1x_2\ldots x_n)$, where $1\leq i\leq n-1$.
\item[(iv)]
$I=I_{(d;a_1,\ldots,a_n)}$ such that $a_1=\dots=a_{n-r}=d$ and $a_{n-r+1}=\ldots=a_n=d-1$, where $0\leq r\leq n$.
\item[(v)]
$I=P_{F_1}P_{F_2}$ such that $\mid F_1\mid=\mid F_2\mid=2$ and $F_1\cap F_2=\emptyset$.
\end{enumerate}
\end{Theorem}
\begin{proof}
First of all we relabel the non-empty subsets $F_{i}$ by $ \vert F_1 \vert \leq \vert F_2 \vert \leq\ldots \leq \vert F_d \vert $.\\
  $(\Longleftarrow)$ The case $(i)$ is obvious.  Let consider the cases $(ii)$ and $(v)$. Then by using the Alexander dual of $I$ and Theorem \ref{T0} we conclude that $I^{\vee}$ has almost linear resolution and so $I$ is aCM. Thus we have the result in these cases.
For case $(iii)$, we have $I=(x_1)\ldots(x_i)\ldots(x_{n-1})(x_i,x_n)$ where $1\leq i\leq n-1$. Hence $\dim R/I=n-1$. Since $G_I$ has $n-1$ connected components, by Theorem \ref{T6} we have $\depth(R/I)=n-2$. Thus $I$ aCM in this case.
Let consider the case $(iv)$. We assume that $I=I_{(d,a_1,\ldots,a_n)}$ such that $a_1=\dots=a_{n-r}=d$ and $a_{n-r+1}=\ldots=a_n=d-1$,  where $0\leq r\leq n$.
Therefore, by Theorem \ref{T7}, $\Ass(I)=\{P_{F_1},\ldots,P_{F_{d-1}},\frk{m}\}$ where $P_{F_i}=(x_1,\ldots,\hat{x}_{n-r+i},\ldots,x_n)$. Hence $\dim R/I=1$ and $\depth R/I=0$. Therefore $I$ is aCM.

$(\Longrightarrow)$ By Theorem \ref{T5}, $\dim(R/I)=n- \min \{ |F_i | \mid i=1, 2,\ldots , d \}$. Therefore we have $\dim(R/I)=n- |F_1|$. Also by Theorem \ref{T6}, we have $ \depth(R/I)= c(G_{I})-1+n -|\cup_{i=1}^{d} F_{i} |$. Since $I$ is  full-supported transversal polymatroidal ideal, $I$ is aCM if and only if $c(G_{I})\geq n- |F_{1}|$. Since $n \geq c(G_{I}) |F_1|$ it follows that $c(G_{I})=n- |F_{1}|\geq (c(G_{I})-1 )|F_1|$. Therefore this inequality is valid if and only if  $(1) c(G_{I})=1$, $(2) |F_1|=1$ or $(3) |F_{1}|=2$ and $n=4 $.\\
Assume $c(G_{I})=1$, then $|F_{1}|\geq n-c(G_{I})=n-1$. If $|F_{1}|=n$, then $I$ is a Veronese ideal. Now, we assume that $|F_{i}|=n-1$ for $1\leq i\leq r$ and $|F_{i}|=n$  for $r+1\leq i\leq d$. Since $I$ is a full-supported transversal polymatroidal ideal without loss of generality we may assume that $I=P_{F_1} P_{F_2}\ldots  P_{F_{r}} \frk{m}^{d-r}$ such that $P_{F_i}=(x_1,\ldots,{\hat{x}}_{n-r+i},\ldots,x_n)$. Therefore by Corollary \ref{C3}, $I=P_{F_1} \cap P_{F_2} \cap\ldots\cap P_{F_{r}} \cap \frk{m}^{d}$ and by \cite[Proposition 2.11]{BH}, $I$ is Veronese type ideal of the form $I=I_{(d; a_1,\ldots, a_n)}$ such that $a_1=\ldots=a_{n-r}=d$ and $a_{n-r+1}=\ldots=a_n=d-1$.
If $|F_1|=1$, then $c(G_{I})\geq n-1$. If $c(G_{I})=n$, then $I$ is a principal ideal. If $c(G_{I})=n-1$, then the number of connected components of the graph $G_I$ are $n-1$ and so we have two cases for considering.
{\bf Case1}: $P_{F_i}=(x_i)$ for $1\leq i\leq n-2$ and $P_{F_{n-1}}=(x_{n-1},x_n)$. Since $I$ is a transversal polymatroidal ideal, we have $I=(x_1)(x_2)\ldots(x_{n-2})(x_{n-1},x_n)$. Therefore $I=(x_1\ldots x_{n-2}x_{n-1},x_1\ldots x_{n-2}x_n)$.\\
{\bf Case 2}: $P_{F_i}=(x_i)$ for $1\leq i\leq n-1$ and $P_{F_{n}}=(x_{i},x_n)$. Thus
\[I=(x_1\ldots x_{i-1}x_i^2x_{i+1}\ldots x_{n-1},x_1\ldots x_n).\]
If $|F_1|=2$, then $c(G_{I}) \geq 2c(G_{I})-2$. Therefore $c(G_{I})\leq 2$ and $n\leq 4$. If $c(G_{I})=1$, then the result follows as above. If $c(G_{I})=2$, then $n=4$. Since $|F_1|=2$, we have $I=P_{F_1}P_{F_2}$ such that $|F_1|=|F_2|=2$ and $F_{1} \cap F_{2} =\emptyset$.\\
  If $|F_1|\geq 3$, then $c(G_{I}) =1$. Hence the result follows as above in this case.
\end{proof}

\section{The aCM simplicial complex}
We say that a simplicial complex $\Delta$ is almost pure if for every two facets $F_{i}$ and $F_j$ belonging to $\Delta$, one has $\vert \vert F_{i} \vert - \vert F_j \vert \vert \leq 1$.
\begin{Lemma}\label{L1}
Let $\Delta$ be an aCM simplicial complex. Then it is almost pure.
\end{Lemma}
\begin{proof}
Set $\Delta =\langle G_1, G_2, \ldots , G_m \rangle$, where $G_1, \ldots , G_m$ are facets of $\Delta$. Recall that the Stanley-Reisner ideal of $\Delta$ has the presentation $I=I_{\Delta}=P_{F_1} \cap P_{F_2} \cap \ldots \cap P_{F_m}$ such that $F_i=\bar{G_i}$, by \cite[Lemma 1.5.4]{HH1}. The minimal prime ideals of $I_{\Delta}$ correspond to the facets
of $\Delta$. Hence $\Delta$ is almost pure if and only if $\bight(I)-\hte(I) \leq 1$. Since $\Delta$ is aCM, we have the result by Corollary \ref{C1}.
\end{proof}

Let $\Delta$ be a simplicial complex on $[n]$ of dimension $d -1$. Recall that, for each $0 \leq i \leq d-1$, the simplicial complex $\Delta^{(i)}:=\lbrace F \in \Delta \vert ~ \vert F \vert \leq i+1 \rbrace$ is called the $i$-th skeleton of $\Delta$.  In addition, for each $0 \leq i \leq d-1$, the pure $i$-th skeleton of $\Delta$ is defined to be the pure subcomplex $\Delta(i)$ of $\Delta$ whose facets are those faces $F$ of $\Delta$ with $\vert F \vert = i+1$ (see \cite[Section 8.2.6]{HH1}).

The simplicial complex $\Delta$ is said to be connected if there exists a sequence of facets $F = F_0, F_1,\ldots, F_{n-1}, F_n = E$ such that $F_i\cap F_{i+1 } \ne \emptyset$ (see \cite[Section 1.5.1]{HH1}).

In the following we use $\tilde{H}_{i}(\Delta; k)$ which is the ith reduced simplicial homology group of $\Delta$ with coefficients in $k$.

\begin{Lemma}\label{L2}
Let $\Delta$ be $(d-1)$-dimensional ($d \geq 3$) simplicial complex such that for all faces $F \in \Delta$ and $i < \dim lk_{\Delta}F-1$, one has $\tilde{H}_{i}(lk_{\Delta}F; k)=0$. Then $\Delta$ is connected.
\end{Lemma}
\begin{proof}
It is clear that $\tilde{H}_{i}(lk_{\Delta}{\emptyset}; k)=0$ for $i < \dim lk_{\Delta}{\emptyset}-1$, as $\emptyset$ is a face of $\Delta$. Therefore $\tilde{H}_{i}({\Delta}; k)=0$ for  $i < \dim{\Delta}-1$.
Since $ \dim{\Delta} \geq 2$, one has $\tilde{H}_{0}({\Delta}; k)=0$. Now by applying \cite[Proposition 6.2.3]{V}, we have $\Delta$ is connected.
\end{proof}

\begin{Example}
Let $\Delta= \langle \lbrace 1,2 \rbrace , \lbrace 4,5 \rbrace , \lbrace 3 \rbrace \rangle$ be $1$-dimensional simplicial complex on $[5]$. Then $I_{\Delta}=(x_1,x_2,x_3)\cap (x_3,x_4,x_5) \cap(x_1,x_2,x_4,x_5)$.
Set $I=I_{\Delta}$. Then \[I=(x_{1}x_{3}, x_{1}x_{4}, x_{1}x_{5}, x_{2}x_{3}, x_{2}x_{4}, x_{2}x_{5}, x_{3}x_{4}, x_{3}x_{5} )\] is aCM but $\Delta$ is not connected.
\end{Example}

\begin{Theorem}\label{T4}
A simplicial complex $\Delta$ of dimension $d-1$ is aCM if and only if

\item[(i)] $\tilde{H}_{i}(lk_{\Delta}F; k)=0$, for all $F \in \Delta \setminus \Delta(d-1)$, $i < \dim lk_{\Delta}F$ and
\item[(ii)] $\tilde{H}_{i}(lk_{\Delta}F; k)=0$, for all $F\in \Delta(d-1)$, $i < \dim lk_{\Delta}F-1$.

\end{Theorem}
\begin{proof}
Let $\Delta$ be aCM, by Hochster's formula $\tilde{H}_{i-\vert F \vert -1}(lk_{\Delta}F; k)=0$ for all $F \in \Delta$ and $i < d-1$. If $F \in \Delta$, then there is a face $F_{1}$ of dimension $d-1$ or $d-2$ containing $F$ such that $\dim lk_{\Delta}F=\vert F_{1}\setminus F \vert -1$, since $ F_{1} \setminus F \in lk_{\Delta}F$ and  $\Delta$ is almost pure by Lemma \ref{L1}. Therefore $\dim lk_{\Delta}F=d- \vert F \vert -1$ for $F \in \Delta(d-1)$ and $\dim lk_{\Delta}F=d- \vert F \vert -2$ for $F \in \Delta \setminus \Delta(d-1)$. Hence $\tilde{H}_{i}(lk_{\Delta}F; k)=0$, for all $F \in \Delta \setminus \Delta(d-1)$, $i < \dim lk_{\Delta}F$ and for all $F\in \Delta(d-1)$, $i < \dim lk_{\Delta}F-1$.\\
Conversely, it is enough to show that $\Delta$ is almost pure simplicial complex. Indeed, if $\Delta$ be almost pure, then $\dim lk_{\Delta}F=d- \vert F \vert -1$ or $\dim lk_{\Delta}F=d- \vert F \vert -2$ for all $F \in \Delta$. Therefore $\tilde{H}_{i-\vert F \vert -1}(lk_{\Delta}F; k)=0$ for $F \in \Delta$ and $i < d-1$.
We proceed by induction on the dimension of $\Delta$. If $\dim (\Delta)=1$, then $\Delta$ is almost pure. Assume $d \geq 3$ and  $F= \{ x_l \}$ be a face of $\Delta$ such that $x_l$ is a vertex. If $F \in \Delta(d-1)$, then by induction hypothesis $lk_{\Delta}F$ is almost pure, i.e, if $G$ be a facet of $lk_{\Delta}F$, then $|G|=d-2$ or $|G|=d-3$, since $\dim lk_{\Delta}F=d-2$. Set $\Gamma:=lk_{\Delta}F$. If $F \in \Delta \setminus \Delta(d-1)$, then by hypothesis for all $G \in \Gamma$ and for all $i < \dim lk_{\Gamma}G$, $\tilde{H}_{i}(lk_{\Gamma}G; k)=0$, since $G \cup F \in \Delta \setminus \Delta(d-1)$ and $lk_{\Gamma}G=lk_{\Delta}F\cup G$. Hence by \cite[Theorem 6.3.12]{V}, $lk_{\Delta}F$ is CM. Thus if $F_i$ and $F_j$ be a facets of $\Delta$ such that have a vertex $F= \{ x_l\}$ in common , then cardinality of $F_i$ and $F_j$ is equal $d$ or $d-1$ for $F \in \Delta(d-1)$ and $|F_i|=|F_j|$ for $F \in \Delta \setminus \Delta(d-1)$. let $H$ and $E$ be two facets of $\Delta$. Since $\Delta$ is connected, there exist facets $F_1, . . . , F_r$ with $H = F_1$ and $E = F_r$ such that $F_i \cap F_{i+1} \ne \emptyset$ for $i = 1, . . . , r$. Since for each $i$, $F_i$ and $F_{i+1}$ have a vertex in common, it follows that the cardinality of $F_i$ and $F_j$ is equal to $d$ or $d-1$. In particular $|H|$ and $|K|$ is equal to $d$ or $d-1$. Hence $\Delta$ is almost pure.
\end{proof}

\begin{Corollary}
Let $\Delta$ be $2$-dimensional simplicial complex. Then $\Delta$ is connected if and only if $\Delta$ be aCM.
\end{Corollary}
\begin{proof}
$(\Leftarrow)$ This follows by Lemma \ref{L2} and Theorem \ref{T4}.\\
$(\Rightarrow)$ Let $\Delta$ be connected. If $\vert F \vert=2$, then $lk_{\Delta}F$ consists of a discrete set of vertices or $lk_{\Delta}F$ is empty set . Therefore $\dim lk_{\Delta}F\leq 0$. Thus $F$ satisfies the condition of Theorem \ref{T4}. If
 $\vert F \vert=1$, then $lk_{\Delta}F$ consists of faces of dimension one or discrete set of vertex or both of them. Therefore $\dim lk_{\Delta}F=0$ for $F\in \Delta\setminus \Delta(2)$ and  $\dim lk_{\Delta}F=1$ for $F\in\Delta(2)$. Hence in this case $F$ satisfies the condition  of Theorem \ref{T4}. Since $\Delta$ is connected, we have $\dim \tilde{H}_{0}({\Delta}; k)=0$. Since $\tilde{H}_{0}({\Delta}; k) $ is free $k$-module of rank $0$ (see \cite[Proposition 6.2.3]{V}), we have $\tilde{H}_{0}({\Delta}; k)=0$. Hence $\tilde{H}_{0}({lk_{\Delta}} \emptyset; k)=0$.
\end{proof}

\begin{Corollary}
Let $\Delta$ be an aCM simplicial complex and F be a face of $\Delta$. Then $lk_{\Delta}F$ is aCM. In particular, if $F \in \Delta \setminus \Delta(d-1)$, then $lk_{\Delta}F$ is CM.
\end{Corollary}
\begin{proof}
Set $\Gamma:=lk_{\Delta}F$ and let $G$ be a face of $\Gamma$. Since $lk_{\Gamma}G= lk_{\Delta}F\cup G$ and $\Delta$ is aCM, we have $\tilde{H}_{i}(lk_{\Delta}F\cup G ;k)=0$ for $F\cup G \in \Delta \setminus \Delta(d-1)$, $i < \dim lk_{\Delta}F\cup G$ and for $F\cup G \in \Delta(d-1)$, $i < \dim lk_{\Delta}F\cup G -1$. If $F\cup G \in \Delta \setminus \Delta(d-1)$, then $ G \in \Gamma \setminus \Gamma(d- |F|-1)$. Thus by Reisner's criterion $\Gamma$ is CM in this case.
If $F\cup G \in \Delta(d-1)$, then $ G \in \Gamma(d- |F|-1)$ and so $\tilde{H}_{i}(lk_{\Gamma}G ;k)=0$ for all $ G \in \Gamma(d- |F|-1)$, $i < \dim lk_{\Gamma} G -1$. Hence by Theorem \ref{T4}, $lk_{\Delta}F$ is aCM .
\end{proof}


\begin{acknowledgement}
We would like to thank deeply grateful to the referee for the careful reading
of the manuscript and the helpful suggestions. The second author has been supported
financially by Vice-Chancellorship of Research and Technology, University of Kurdistan
under research Project No. 99/11/19299.
\end{acknowledgement}



\begin{thebibliography}{}
\bibitem{BH}
S. Bandari and J. Herzog, {\it Monomial localizations and polymatroidal ideals}, Eur. J. Comb.,
{\bf 34}(2013), 752-763.

\bibitem{CTT}
L. Chu, Z. Tang and H. Tang, {\it A note on almost Cohen-Macaulay modules}, J. Algebra Appl., {\bf 14}
(2015) 1550136.
\bibitem{ER}
J. Eagon and V. Reiner, {\it Resolutions of Stanley-Reisner rings and Alexander duality}, J. Pure Appl.
Algebra, {\bf 130}(1998), 265-275.
\bibitem{G}
D. R. Grayson and M. E. Stillman, {\it Macaulay 2, a software system for research in algebraic geometry},
Available at http://www.math.uiuc.edu/Macaulay2/.

\bibitem{Ha}
Y. Han, {\it D-rings}, Acta Math. Sinica., {\bf 4}(1998), 1047-1052.

\bibitem{HH2}
J. Herzog and T. Hibi, {\it Discrete polymatroids}, J. Algebraic Combin., {\bf 16}(2002), 239-268.

\bibitem{HH3}
J. Herzog and T. Hibi, {\it Cohen-Macaulay polymatroidal ideals}, Eur. J. Comb., {\bf 27}(2006), 513-517.

\bibitem{HH1}
J. Herzog and T. Hibi, {\it Monomial ideals}, Grad.Texts Math., vol.{\bf 260}, Springer-Verlag London, Ltd., London, (2011).

\bibitem{HRV}
J. Herzog, A. Rauf and M. Vladoiu, {\it The stable set of associated prime ideals of a polymatroidal
ideal}, J. Algebr Comb., {\bf 37}(2013), 289-312.
\bibitem{HV}
J. Herzog and M. Vladoiu, {\it Monomial ideals with primary components given by powers of monomial prime ideals}, Electron. J. Combin., {\bf 21} (2014), P1.69.

\bibitem{HT}
L. T. Hoa and N. V. Trung, {\it On the Castelnuovo-Mumford regularity and the arithmetic degree of monomial ideals}, Math. Z., {\bf 229}(1998), 519-537.

\bibitem{I}
C. Ionescu, {\it More properties of almost Cohen-Macaulay rings}, J. Comm. Algebra, {\bf 7}(2015)
363-372

\bibitem{K1}
M. Kang, {\it Almost Cohen-Macaulay modules}, Comm. Algebra, {\bf 29}(2001), 781-787.

\bibitem{K2}
M. Kang,{\it Addendum to almost Cohen-Macaulay modules}, Comm. Algebra, {\bf 30}(2002), 1049-1052.

\bibitem{MN}
A. Mafi and D. Naderi,{\it A note on linear resolution and polymatroidal ideals}, Proc. Math. Sci., {\bf 131}(2021).
https://doi.org/10.1007/s12044-021-00620-z.
\bibitem{MT}
A.Mafi and S. Tabejamaat, {\it Results on almost Cohen-Macaulay modules}, J. Algebraic Syst., {\bf 3} (2016),
147-150.
\bibitem{TM}
S. Tabejamaat and A. Mafi,{\it About a Serre-type condition for modules}, J. Algebra Appl., {\bf 16}
1750206.
\bibitem{TMA}
S. Tabejamaat, A. Mafi and Kh. Ahmadi Amoli,{\it  Property of Almost Cohen-Macaulay over Extension
Modules}, Algebra Colloquium, {\bf 24} (2017), 509-518.
\bibitem{Te}
N. Terai, {\it Alexander duality theorem and Stanley-Reisner rings},~S$\ddot{u}$rikaisekikenky$\ddot{u}$sho~K$\ddot{o}$ky$\ddot{u}$ruko, {\bf 1078}(1999), 174-184, Free resolutions of coordinate ringsof projective varieties and relatedtopics (Kyoto 1998).
\bibitem{V}
R. H. Villarreal, {\it Monomial Algebras}, Monographs and Research Notes in Mathematics, Chapman
and Hall/CRC, (2015).
\bibitem{VL}
M. Vladoiu, {\it Equidimensional and unmixed ideals of Veronese type}, Comm. Algebra, {\bf 36}(2008), 3378-
3392.

\end{thebibliography}
\end{document}